\documentclass[oneside,a4paper]{amsart}

\usepackage[T1]{fontenc}
\usepackage[utf8]{inputenc}
\usepackage{amsmath, amsthm, amsfonts, amssymb}

\usepackage[colorlinks=true,hyperindex, linkcolor=magenta, pagebackref=false, citecolor=cyan,pdfpagelabels]{hyperref}
\usepackage{url}
\usepackage[all]{xy}
\usepackage{placeins} 
\usepackage[usenames, dvipsnames]{color} 
\usepackage{euscript} 
\usepackage{bbm} 
\usepackage{enumitem} 
\usepackage[nameinlink,capitalise,noabbrev]{cleveref} 
\usepackage{todonotes}

\usepackage[usenames,dvipsnames]{color}
\usepackage{tikz-cd}
\usepackage{tikz}
\usetikzlibrary{arrows,backgrounds,
    decorations.pathreplacing,
    decorations.pathmorphing
}
\usetikzlibrary{matrix,arrows}

\newcommand{\euscr}[1]{\EuScript{#1}} 
\newcommand{\Ext}{\textnormal{Ext}} 
\DeclareMathOperator*{\colim}{colim}

\newcommand{\Mod}{\euscr{M}od} 
\newcommand{\pushout}{\arrow[ul, phantom, "\ulcorner", very near start]}

\def\MU{\mathrm{MU}}

\theoremstyle{plain}

\newtheorem{theorem}{Theorem}
\newtheorem{lemma}[theorem]{Lemma}

\newtheorem{corollary}[theorem]{Corollary}

\newtheorem*{theorem*}{Theorem}

\theoremstyle{definition}

\newtheorem{definition}[theorem]{Definition}
\newtheorem{remark}[theorem]{Remark}

\newtheorem*{remark*}{Remark}
\newtheorem*{interpretation*}{Interpretation}
\newtheorem*{definition*}{Definition}
\newtheorem*{conjecture*}{Conjecture}
\newtheorem*{notation*}{Notation}
\newtheorem*{convention*}{Convention}

\theoremstyle{remark}

\numberwithin{equation}{section}

\makeatletter
  \def\subsection{\@startsection{subsection}{1}%
  \z@{.7\linespacing\@plus\linespacing}{.5\linespacing}%
  {\normalfont\bfseries\centering}}
\makeatother

\let\oldtocsection=\tocsection
\let\oldtocsubsection=\tocsubsection
\let\oldtocsubsubsection=\tocsubsubsection
\renewcommand{\tocsection}[2]{\hspace{0em}\oldtocsection{#1}{#2}}
\renewcommand{\tocsubsection}[2]{\hspace{1em}\oldtocsubsection{#1}{#2}}
\renewcommand{\tocsubsubsection}[2]{\hspace{2em}\oldtocsubsubsection{#1}{#2}}

\calclayout

\begin{document}
\title[]{Adams-type maps are not stable under composition}

\author{Robert Burklund}
\address{Department of Mathematics, MIT, Cambridge, MA, USA}
\email{burklund@mit.edu}

\author{Ishan Levy}
\address{Department of Mathematics, MIT, Cambridge, MA, USA}
\email{ishanl@mit.edu}

\author{Piotr Pstr\k{a}gowski}
\address{Department of Mathematics, Harvard, Cambridge, MA, USA}
\email{pstragowski.piotr@gmail.com}

\begin{abstract}
We give a simple counterexample to the plausible conjecture that Adams-type maps of ring spectra are stable under composition. We then show that over a field, this failure is quite extreme, as any map of $\mathbb{E}_{\infty}$-algebras is a transfinite composition of Adams-type maps. 
\end{abstract}

\maketitle 



The Adams spectral sequence is a fundamental tool in stable homotopy theory which, given a map $A \to B$ of ring spectra, lets us compute the homotopy groups of $A$ in terms of information living over $B$. 
Unfortunately, the $E_{2}$-page of the Adams spectral sequence can be difficult to identify in general. 
The standard additional assumption which makes the $E_2$-page computable in terms of homological algebra is that the map $A \to B$ is flat in one of the senses below.

\begin{definition} 
We say a map $A \rightarrow B$ of $\mathbb{E}_{1}$-rings is 
\begin{enumerate}
    \item \emph{Adams-type} if we can write $B \simeq \varinjlim B_{\alpha}$ as a filtered colimit of perfect $A$-modules with the property that $\pi_{*}(B \otimes_{A} B_{\alpha})$ is projective as a $\pi_*B$-module, 
    \item \emph{descent-flat}\footnote{In older literature, what we call descent-flat is often simply referred to as \emph{flat}. We avoid the latter term as it is also often used to refer to maps such that $\pi_*A \rightarrow \pi_*B$ is flat, a much stronger condition than the one we work with.} if $\pi_{*}(B \otimes_{A} B)$ is flat as a left $\pi_*B$-module. 
\end{enumerate}
\end{definition}

If $A \rightarrow B$ is descent-flat, then for any pair $M$, $N$ of $A$-modules such that $\pi_{*}(B \otimes_{A} N)$ is projective $\pi_*B$-module, the associated Adams spectral sequence computing homotopy classes of $A$-module maps from $M$ to $N$ has signature
\[
\Ext^{s, t}_{(\pi_*B,\, \pi_*(B \otimes_A B))} \left(\pi_{*}(B \otimes_{A} N),\, \pi_{*}(B \otimes_{A} M)\right) \Longrightarrow [N,\, M]_{A}.
\]
If $A \rightarrow B$ is furthermore Adams-type, then one can construct an Adams spectral sequence with the above signature with no projectivity assumption on $\pi_{*}(B \otimes_{A} N)$, as in the work of Devinatz \cite[\S 1]{dev_morava}.

As both of the above notions are a form of flatness, it is natural to expect that they are stable under composition, and the authors of this note have spent considerable amount of time trying to show that this is indeed the case. To our surprise, this is very far from being true; our first result is the following simple counterexample. 

\begin{theorem}
\label{theorem:main_theorem}
Both of the maps $\mathbb{S} \rightarrow \MU$ and $\MU \rightarrow \mathbb{Z}$ are Adams-type, but the composite $\mathbb{S} \rightarrow \mathbb{Z}$ is not even descent-flat. In particular, neither Adams-type nor descent-flat maps of $\mathbb{E}_{\infty}$-ring spectra are stable under composition.
\end{theorem}

\begin{proof}
It is well-known that $\mathbb{S} \rightarrow \MU$ is Adams-type \cite[13.4(iv)]{adams1995stable}. Similarly, $\mathbb{Z}_{*} \mathbb{Z}$ is known to contain torsion elements in positive degrees and so is not flat over $\mathbb{Z}_{*}$ \cite{cartan1955seminaire}
, hence $\mathbb{S} \rightarrow \mathbb{Z}$ is not descent-flat. 

To see that $\MU \to \mathbb{Z}$ is Adams-type, note that if we write $\MU_{*} \simeq \mathbb{Z}[b_{1}, b_{2}, \ldots]$ for some choice of generators $b_{i}$, then 
\[
\mathbb{Z} = \colim_{n} \ \bigotimes^{\MU}_{1 \leq i \leq n}\MU/b_i
\]
in $\MU$-modules. We claim this is the needed expression of $\mathbb{Z}$ as a filtered colimit. To see this, note that we have an equivalence 
\[
\mathbb{Z} \otimes_{\MU} \MU/b_{i} \simeq \mathbb{Z}/b_{i} \simeq \mathbb{Z} \oplus \Sigma^{|b_{i}|+1} \mathbb{Z}
\]
as $b_{i}$ vanishes in $\mathbb{Z}$ and hence for each $n$ we have 
\[
\mathbb{Z} \otimes_{\MU} \bigotimes^{\MU}_{1 \leq i \leq n}\MU/b_i \simeq \bigotimes^{\mathbb{Z}}_{1 \leq i \leq n} \mathbb{Z} \oplus \Sigma^{|b_{i}|+1} \mathbb{Z}
\]
which is a perfect $\mathbb{Z}$-module whose homotopy groups form a finitely generated free abelian group, as needed. 
\end{proof}




The counterexample of \cref{theorem:main_theorem} was contrary to our expectations, and led us to ask which maps of ring spectra can be written as compositions (perhaps infinite) of Adams-type maps. In \cref{theorem:a_map_of_commutative_k_algebras_a_tf_composite_of_adams_type_maps} below we show that, surprisingly, at least for $\mathbb{E}_{\infty}$-algebras over a field, \emph{any} map of ring spectra has this property.


\begin{lemma}\label{lem:adamstypestability}
Let $R \to R'$ be a map of $\mathbb{E}_{2}$-algebras. If $A$ is an $\mathbb{E}_{1}$-$R$-algebra such that the unit $R \rightarrow A$ is of Adams-type, then so is $R' \rightarrow R' \otimes_{R} A$. 
\end{lemma}

\begin{proof}
    We start with an observation about the Adams-type condition:
    If $B \to C$ is a map of $\mathbb{E}_1$-algebras, and $M$ is a $B$-module, then 
    $\pi_*(C \otimes_B M)$ is projective as a $\pi_*C$-module iff the $C$-module $C \otimes_B M$ is a retract of a sum suspensions of copies of $C$. We call a $B$-module which satisfies this condition $C$-projective.
    
    Now we prove the lemma.
    Write $A$ as a filtered colimit $A \simeq \colim A_{\alpha}$ of $A$-projective $R$-modules.
    Then $R' \otimes_{A} R \simeq \colim R' \otimes_{R} A_{\alpha}$ is a filtered colimit with the required property, as 
    \[ (R' \otimes_R A) \otimes_{R'} (R' \otimes_R A_\alpha) \simeq R' \otimes_R (A \otimes_R A_\alpha) \]
    which implies that $R' \otimes_R A_\alpha$ is $(R' \otimes_R A)$-projective.
\end{proof}

\begin{corollary} \label{cor:po-at}
Adams-type maps of $\mathbb{E}_{\infty}$-ring spectra are stable under pushouts. 
\end{corollary}


\begin{lemma}
\label{lemma:adams_type_if_source_or_target_is_a_field}
Let $A \rightarrow B$ be a map of $\mathbb{E}_{1}$-rings such that either $\pi_{*}A$ or $\pi_{*}B$ is a graded skew-field (that is, every non-zero homogeneous element is invertible). Then $A \rightarrow B$ is Adams-type. 
\end{lemma}

\begin{proof}
Suppose that the first condition holds. Then, every $A$-module is a direct sum of shifts of $A$, and hence for any way to write $B \simeq \varinjlim B_{\alpha}$ as a filtered colimit of perfect $A$-modules, $B \otimes_{A} B_{\alpha}$ is a finite direct sum of shifts of $B$ and hence $\pi_{*}(B \otimes_{A} B_{\alpha})$ is a free $B_{*}$-module. 

Similarly, if $\pi_{*} B$ is a graded skew-field, then the perfect $B$-modules $B \otimes_{A} B_{\alpha}$ are also finite sums of shifts of $B$, as needed.
\end{proof}

\begin{theorem}
\label{theorem:a_map_of_commutative_k_algebras_a_tf_composite_of_adams_type_maps}
Let $k$ be a field and $A \rightarrow B$ a map of $\mathbb{E}_{\infty}$-$k$-algebras. Then any $A \rightarrow B$ can be factored as transfinite composite of Adams-type maps.
\end{theorem}

\begin{proof}
Let $k \{ x_{n} \}$ denote the free $\mathbb{E}_{\infty}$-$k$-algebra on a variable of degree $| x_{n} | = n$.  Let $\mathcal{G}$ denote the set of maps $i_{n}: k \rightarrow k \{ x_{n} \}$ together with the maps $p_{n}: k \{ x_{n} \} \rightarrow k$ determined by $x_{n} \mapsto 0$. As a consequence of \cref{lemma:adams_type_if_source_or_target_is_a_field}, all of the $i_{n}$ and $p_{n}$ are Adams-type, as either the source or target is a field. By the small object argument, the given map $A \rightarrow B$ can be factored as 
\[
A \rightarrow B^{\prime} \rightarrow B
\]
where $B^{\prime} \rightarrow B$ has the right lifting property with respect to the maps in $\mathcal{G}$, and $A \rightarrow B^{\prime}$ is a transfinite composition of maps $A_{\alpha} \to A_{\alpha+1}$ which fit into pushout squares 
\begin{center}
\begin{tikzcd}
	S \ar[r,"f"]\ar[d]& \ar[d]T \\
	A_{\alpha}\ar[r] & A_{\alpha+1}\pushout
\end{tikzcd}
\end{center}
where $f \in \mathcal{G}$. Each such map is Adams-type by \cref{cor:po-at}, and it follows that $A \rightarrow B^{\prime}$ is a transfinite composite of such maps. On the other hand, one may observe  that the right lifting property with respect to $i_n$ implies surjectivity on $\pi_n(-)$ and the right lifting property with respect to $p_n$ implies injectivity on $\pi_n(-)$, therefore  $B^{\prime} \rightarrow B$ is an equivalence.
    \qedhere
\end{proof}

\begin{remark}
By using a more careful argument, where we also allow pushouts along maps $k\{S\} \to k$ and $k \to k\{S\}$ for $S$ a (graded) set of generators, one can show that in the context of \cref{theorem:a_map_of_commutative_k_algebras_a_tf_composite_of_adams_type_maps} the given map factors as an $\omega$-indexed composite of Adams-type maps. 
\end{remark}

We believe the above two results show that the somewhat unexplored theory of Adams-type ring spectra still has a few surprises up its sleeve. To further emphasize this point, we share with the reader a few natural questions which we believe are open.

\begin{enumerate}
    \item Does there exist a descent-flat map $A \rightarrow B$ which is not Adams-type\footnote{If we had to guess, we expect that a descent-flat map of ring spectra which is not Adams-type does indeed exist, although we couldn't find one. A curious variant of this question (which we also do not know the answer to) would be to ask if exists a descent-flat map $A \rightarrow B$ such that the associated homology theory $\pi_{*}(B \otimes_{A} -): \Mod(A) \rightarrow \euscr{C}\mathrm{omod}_{\pi_{*}(B \otimes_{A} B)}$ is adapted in the sense of \cite[Definition 2.19]{patchkoria2021adams}, which is not Adams-type? In other words, does possessing a modified Adams spectral sequence based on $\pi_{*}(B \otimes_{A} B)$-comodules characterize Adams-type maps?}?
    \item Is every map of $\mathbb{E}_{1}$-algebras (over the sphere) a transfinite composite of Adams-type maps? Is it true for every map of $\mathbb{E}_{\infty}$-algebras? 
    \item Let $A \rightarrow B$ be an Adams-type map and $M$ be an $A$-module such that $\pi_{*}(B \otimes_{A} M)$ is a flat $B_{*}$-module. Can we write $M \simeq \varinjlim M_{\alpha}$ as a filtered colimit of perfect $A$-modules such that $\pi_{*}(B \otimes_{A} M_{\alpha})$ is a projective $B_{*}$-module?
\end{enumerate}

\bibliographystyle{amsalpha}
\bibliography{adamstype}

\end{document}